\documentclass[12pt]{article}
\usepackage{mathrsfs}
\usepackage{amsmath, amsthm, amssymb}
\usepackage{graphicx}

\begin{document}
\bibliographystyle{unsrt}
\title{Poisson--Lie Sigma Models on Drinfel'd double}
\author{Jan Vysok\'y, Ladislav Hlavat\'y}
\maketitle
{Czech Technical University in Prague, Faculty of Nuclear Sciences
and Physical Engineering,  B\v rehov\'a 7, 115 19 Prague 1, Czech
Republic}

\newtheorem{thm}{Theorem}[section]
\newtheorem{lem}[thm]{Lemma}

\newtheorem{tvrzx}[thm]{Proposition}
\newenvironment{tvrz}{\begin{tvrzx}}{\smallskip\end{tvrzx}}

\newtheorem{corx}[thm]{Corollary}
\newenvironment{cor}{\begin{corx}}{\smallskip\end{corx}}

\newtheorem{lemmax}[thm]{Lemma}
\newenvironment{lemma}{\begin{lemmax}}{\smallskip\end{lemmax}}

\newtheorem{theoremx}[thm]{Theorem}
\newenvironment{theorem}{\begin{theoremx}}{\smallskip\end{theoremx}}

\theoremstyle{definition}
\newtheorem{definicex}[thm]{Definition}
\newenvironment{definice}{\begin{definicex}}{\medskip\end{definicex}}

\theoremstyle{remark}
\newtheorem{remx}[thm]{Remark}
\newenvironment{rem}{\begin{remx}}{\medskip\end{remx}}

\theoremstyle{definition}
\newtheorem{examplex}[thm]{Example}
\newenvironment{example}{\begin{examplex}}{\medskip\end{examplex}}

\newcommand\ts{\rule{0pt}{2.6ex}}
\newcommand\bs{\rule[-1.2ex]{0pt}{0pt}}

\def\R{\mathbb{R}}
\def\C{\mathbb{C}}
\def\N{\mathbb{N}}
\def\Z{\mathbb{Z}}
\def\T{\mathbb{T}}
\def\P{\mathcal{P}}
\def\Q{\mathcal{Q}}
\def\<{\langle}
\def\>{\rangle}
\def\~{\widetilde}
\def\^{\wedge}
\def\g{\mathfrak{g}}
\def\d{\mathfrak{d}}
\def\gxg{\mathfrak{g} \otimes \mathfrak{g}}
\def\X{\mathcal{X}}
\def\Pt{\mathbf{P}}
\def\io{\mathit{i}}

\def\ddt{\left. \frac{d}{dt}\right|_{t=0} \hspace{-0.5cm}}
\def\ddtnospace{\left. \frac{d}{dt}\right|_{t=0}}
\def\dds{\left. \frac{d}{ds}\right|_{s=0} \hspace{-0.5cm}}
\newcommand{\ddy}[2]{\left. \frac{\partial}{\partial{y^{#1}}}\right|_{\sss{#2}}}
\newcommand{\ddx}[2]{\left. \frac{\partial}{\partial{x^{#1}}}\right|_{\sss{#2}}}
\newcommand{\ddygen}[1]{\frac{\partial}{\partial{y^{#1}}}}
\newcommand{\ddxdole}[2]{\left. \frac{\partial}{\partial{x_{#1}}}\right|_{\sss{#2}}}
\newcommand{\ddXdole}[2]{\left. \frac{\partial}{\partial{X_{#1}}}\right|_{\sss{#2}}}
\newcommand{\ddTdole}[2]{\left. \frac{\partial}{\partial{T_{#1}}}\right|_{\sss{#2}}}
\newcommand{\cffx}[2]{\frac{\partial{#1}}{\partial{x^{#2}}}}

\newcommand{\mind}[2]{#1 \cdots #2}
\newcommand{\ul}[1]{\underline{#1}}
\newcommand{\sch}[1]{[\![#1]\!]}
\renewcommand{\qedsymbol}{$\blacksquare$}

\def\sss{\scriptscriptstyle}

\begin{abstract}
Poisson sigma models represent an interesting use of Poisson
manifolds for the construction of a classical field theory. Their
definition in the language of fibre bundles is shown and the
corresponding field equations are derived using a coordinate
independent variational principle.

The elegant form of equations of motion for so called Poisson-Lie
groups is derived. 

Construction of the Poisson-Lie group corresponding to a given Lie
bialgebra is widely known only for coboundary Lie bialgebras. Using
the adjoint representation of Lie group and Drinfel'd double we show
that Poisson-Lie group can be constructed for general Lie bialgebra.
\end{abstract}

\tableofcontents
\section{Introduction}
A theory of Poisson manifolds gives us a possibility to construct an
interesting field theory (classical). It was first introduced by
Peter Schaller and Thomas Strobl in 1994 in \cite{rakusaci1}. It
turned out that this theory included limit cases of the well-known models,
such as two-dimensional Yang-Mills model or $\mathcal{R}^2$ gravity
model. The basic idea is to use a Poisson manifold as a target,
instead of a Riemannian manifold.

We will in detail review the definition of Poisson sigma model
fields, using the formalism of fibre bundles. This is in great detail discussed in \cite{bojowaldstrobl}.
After the necessary introduction of involved objects, we will proceed and give a
definition of a globally defined action integral of the model.

Usually (e.g. \cite{rakusaci1}, \cite{spanele} or \cite{rakusaci2}) the
equations of motion are derived using the Lagrange formalism of the classical field theory.  
In subsection \ref{sec_psmvar} we will derive the equations of motion starting once again from the
variational principle, after giving a geometric sense to the
variation of the fields.

In subsection \ref{sec_psmlinear} we show one of the simplest non-trivial examples
of a Poisson manifold - linear Poisson structure. We work out (as was already done by many others, see e.g. \cite{spanele}) an intrinsic form of the equations of motion using the involved Lie algebra structures. We intend this as an important but simple example, where the "zero curvature" equation for the field of Poisson sigma model appears.

Poisson-Lie groups constitute a natural combination of Poisson structures with the theory of Lie groups. The Poisson bivector of Poisson-Lie group is then a multiplicative bivector field, allowing the computation of its intrinsic derivative. This intrinsic derivative is in fact a Lie bialgebra structure, $1$-cocycle of particular Lie algebra cohomology.

More interestingly, for given Lie bialgebra there exists a unique Poisson bivector on a corresponding Lie group, so that this Lie bialgebra is its intrinsic derivative \cite{luweinstein}. For coboundary Lie bialgebras there exists a well known explicit formula for the Poisson bivector, called Sklyanin bracket (see e.g. \cite{luweinstein}).

However, there exists a procedure for a general Lie bialgebra, not so widely known at the moment. This is why we present it in the subsection \ref{sec_construction}, not going into details.
The key element is the equivalent description of Lie bialgebra - Manin triple, Lie algebra with distinguished Lie subalgebras and a non-degenerate bilinear ad-invariant symmetric form. Corresponding (connected) Lie group, called Drinfel'd double, allows us to define (mutually dual) Poisson-Lie group structures on its Lie subgroups.

In section \ref{sec_psm} we will in detail discuss the case, where the target Poisson manifold of the model is moreover a Poisson-Lie group. We call such models a Poisson-Lie sigma models. With the help of globally valid equations of motion (with no restrictions on the range of the fields at all) derived in subsection \ref{sec_psmvar}, we can use global frame fields of Lie group, right-invariant vector fields, to derive a intrinsic (coordinate-free) form of equations of motion for Poisson-Lie sigma model. Most interestingly, we again get a "zero curvature" form of equation for the fields of the model.  

This result can be reproduced using the language of Lie algebroid morphisms, which was described in \cite{bojowaldstrobl}. 

In the final subsection we give a simple example of non-linear $2$-dimensional Poisson-Lie sigma model with non-coboundary Lie bialgebra, using the construction of Poisson-Lie group structure on Drinfel'd double. 

\section{Poisson sigma models}
\subsection{Poisson manifolds}\label{poissmf} The aim of this section is to
introduce the elements and notation of the Poisson geometry. Its
central point lies in the study of Poisson manifolds -
differentiable manifolds equipped with a Poisson bracket. The
convenient way to describe such structures is the language of
multivector fields, completely skew-symmetric contravariant tensor
fields on manifolds. This is why we have included a brief
introduction to multivector fields.

\begin{definice}
Let $M$ be a differentiable manifold of dimension $n$. We denote
$\mathcal{T}^{k}_0(M)$ the space of $k$-times contravariant tensor
fields on $M$ and $L_k(M)$ its subspace of completely skew-symmetric
tensor fields. We define the {\bfseries{multivector field algebra}}
as a direct sum
\begin{equation}
L(M) := \bigoplus_{k=-\infty}^{+\infty} {L_k(M)}
\end{equation}
equipped with the exterior product $\^$, where $L_0 \equiv
C^{\infty}(M)$ and $L_k = 0$ for $k < 0$ and $k > n$. Multivectors
lying in the particular $L_k(M)$ subspace are called
{\bfseries{homogeneous}}. The {\bfseries{grade}} of a homogeneous
non-zero element $X$ of some $L_k(M)$ is defined as $k$ and denoted
$|X| = k$.

The elements of $L_k(M)$ in the form of $X_1 \^ \cdots \^ X_k, \
X_i \in L_1(M) \equiv \mathfrak{X}(M)$, are called
{\bfseries{simple}}.
\end{definice}

\begin{definice}
{\bfseries{Schouten-Nijenhuis bracket}} $[\cdot,\cdot]: L(M) \times
L(M) \rightarrow L(M)$ is the $\R$-bilinear mapping defined on
homogeneous simple elements of $L(M)$ as
\begin{equation} \label{def_sb}
[X_1 \^ \cdot \cdot \^ X_n, Y_1 \^ \cdot \cdot \^ Y_m] := \end{equation}
\[
\sum_{i=1}^n \sum_{j=1}^m { (-1)^{i+j} [X_i,Y_j] \^ X_1 \^ \cdot
\cdot \^ \widehat{X_i} \^ \cdot \cdot \^ X_n \^ Y_1 \^ \cdot \cdot
\^ \widehat{Y_j} \^ \cdot \cdot \^ Y_m},
\]
where $\widehat{X_i}$ denotes the omission of the vector field in
the product and $[X_i,X_j]$ is an ordinary commutator of the vector
fields. For $f \in L_0(M)$ and homogeneous $X \in L(M)$ the
Schouten-Nijenhuis bracket is defined as
\begin{equation}
[f,X] := -\mathit{i}_{df}{X},
\end{equation}
\begin{equation}
[X,f] := (-1)^{(|X|+1)}\mathit{i}_{df}{X},
\end{equation}
where $\mathit{i}$ is the common interior product operator
(insertion operator).
\end{definice}
It is obvious that the Schouten-Nijenhuis bracket of two multivector
fields is again a multivector field, it is just the sum of exterior
products of vector fields on $M$.
During a study of the classical Hamiltonian mechanics there
naturally arises an interesting geometric structure, a Poisson
bracket of two observable quantities (functions on a phase space).
We define Poisson bracket on an arbitrary differentiable manifold
$M$.
\begin{definice}
Differentiable manifold $M$ is called a {\bfseries{Poisson
manifold}} if it is equipped by an additional structure $\{ \cdot ,
\cdot \}$ called a {\bfseries{Poisson bracket}}, which is a bilinear
map $C^{\infty}(M) \times C^{\infty}(M) \rightarrow C^{\infty}(M)$
having the following properties:
\begin{equation} \label{pb_antisymmetry}
(\forall f,g \in C^{\infty}(M)) \ (\{f,g\} = - \{g,f\}).
\end{equation}
\begin{equation} \label{pb_jacobi}
(\forall f,g,h \in C^{\infty}(M)) \ (\{f,\{g,h\}\} + \{h,\{f,g\}\} +
\{g,\{h,f\}\} = 0).
\end{equation}
\begin{equation} \label{pb_derivation}
(\forall f,g,h \in C^{\infty}(M)) \ (\{fg,h\} = f\{g,h\} +
\{f,h\}g).
\end{equation}
In the other words, $\{\cdot,\cdot\}$ adds on $C^{\infty}(M)$ the
Lie algebra structure and for every $h \in C^{\infty}(M)$ the map
$\{\cdot,h\}: C^{\infty}(M) \rightarrow C^{\infty}(M)$ lies in
$Der(C^{\infty}(M))$.
\end{definice}

The Poisson bracket on $M$ can be easily encoded into the special
bivector field $\P$ on $M$, called (not surprisingly) a Poisson
bivector on $M$.

The skew-symmetry of the Poisson bracket will lead onto a
skew-symmetry of $\P$ (this is why we talk about a bivector field).
The Leibniz rule (\ref{pb_derivation}) will be satisfied "for free"
by every bivector field $\P$. The only problem arises with the
Jacobi identity for the Poisson bracket. It turns out that it can be
encoded into the words of Schouten-Nijenhuis bracket, which we have
canonically defined on every manifold $M$.

\begin{tvrz}
Let $M$ be a differentiable manifold. Every Poisson bracket
$\{\cdot,\cdot\}$ on $M$ corresponds to the unique bivector field
$\P \in L_2(M)$, such that
\begin{equation} \label{pt_schoutenvanish}
[\P,\P] = 0.
\end{equation}
Conversely, every bivector $\P \in L_2(M)$ satisfying
(\ref{pt_schoutenvanish}) can be used to define a Poisson bracket on
$M$. The bivector field $\P$ is called a {\bfseries{Poisson bivector
(field)}} on $M$.
\end{tvrz}

\subsection{Fibre bundles} For definition of dynamical variables of
Poisson sigma models we shall need the concept of the vector bundle
map. For this reason we will bring in a few definitions essential
for the proper setting. We do not intend to go in details, there
exists a plenty of classical literature on this topic.

Let us denote differential fibre bundle as $E \stackrel{\pi}{\rightarrow} M$
and the set of all its smooth  sections as $\Gamma(M,E)$. The subset of global smooth sections is
denoted as $\Gamma_{G}(M,E)$.

\begin{definice}
Let $E \stackrel{\pi}{\rightarrow} M$, $E' \stackrel{\pi'}{\rightarrow} N$ be two fibre bundles. A pair $(f,g): E \rightarrow E'$ is called the {\bfseries{bundle map}}, if:
\begin{enumerate}
\item $f: M \rightarrow N$ is a smooth map of the base manifolds.
\item $g: E \rightarrow E'$ is a smooth map of the total spaces.
\item Following diagram commutes:
\[ \begin{array}{ccc} E & \stackrel{g}{\longrightarrow} & E' \\ \downarrow_{\pi} & \ & \downarrow_{\pi'} \\ M & \stackrel{f}{\longrightarrow} & N \end{array}\]
\end{enumerate}

If $E$ and $E'$ are vector bundles, we call $(f,g)$ the {\bfseries{vector bundle map}}, if $(f,g): E \rightarrow E'$ is a bundle map and
$g: E \rightarrow E'$ is linear "in the fibres", i.e.:
\begin{equation} (\forall p \in M)(\forall u,v \in \pi^{-1}(p))(\forall \alpha \in \R) \ \big( g(\alpha u + v) = \alpha g(u) + g(v) \big). \end{equation}
\end{definice}

\begin{definice}(see e.g. \cite{nakahara})
Let $N$ and $M$ be differentiable manifolds and $E
\stackrel{\pi}{\rightarrow} M$ be a fibre bundle with the typical
fibre $F$. Let $\varphi: N \rightarrow M$ be a smooth map. Then we
can induce a fibre bundle structure above $N$, called the
{\bfseries{pullback bundle}}, as follows:

Its total space $\varphi^{\ast}(E)$ is defined as
\[ \varphi^{\ast}(E) := \{ (p,u) \in N \times E \ | \ \varphi(p) = \pi(u) \}. \]
Its base manifold is $N$, projection is $\pi': (p,u) \in \varphi^{\ast}(E) \mapsto p \in N$ and typical fibre $F$ remains the same.
It is clear that the fibre $F_{p}$ at $p \in N$ is just the same, as the fibre $F_{\varphi(p)}$ of the original bundle.

\end{definice}

\subsection{Fields, action} Let us suppose we have a $2$-dimensional
differentiable orientable manifold $\Sigma$, called usually the
{\bfseries{worldsheet}}. Let us remark that in general we do not
demand a (pseudo-)metric structure on $\Sigma$. We usually take
$\partial{\Sigma} = \emptyset$ (empty boundary) and we want $\Sigma$
to be such manifold, where integration and the Stokes' theorem have
a good sense. Otherwise we have to impose some (boundary) conditions
on the fields.

Next suppose an $n$-dimensional Poisson manifold $(M,\P)$. Again we do not demand $M$ to be a (pseudo-)Riemannian manifold. The manifold $M$ is called a {\bfseries{target manifold}}.

\begin{definice}
{\bfseries{Dynamical field of Poisson sigma model}} is a vector bundle map $(X,A): T\Sigma \rightarrow T^{\ast}M$, where $T\Sigma$ is the tangent bundle of $\Sigma$ and $T^{\ast}M$ is the cotangent bundle of $M$.

$X: \Sigma \rightarrow M$ is a smooth map of the base manifolds, whereas $A: T\Sigma \rightarrow T^{\ast}M$ is a smooth map of the total spaces.
\end{definice}

\begin{rem}
Let $E \stackrel{\pi}{\rightarrow} M$ be a vector bundle. Its set of global smooth sections $\Gamma_{G}(M,E)$ has a natural $C^{\infty}(M)$-linear structure:
\[ \big( \forall \sigma,\tau \in \Gamma_{G}(M,E) \big) \big( \forall f \in C^{\infty}(M) \big)\big( \forall p \in M \big) \ \big( (f\sigma + \tau)(p) := f(p) \sigma(p) + \tau(p) \big). \]
In particular, we will use this property for the pullback bundle $X^{\ast}(T^{\ast}M)$ and the set $\Gamma_{G}(\Sigma,X^{\ast}(T^{\ast}M))$.
\end{rem}

\begin{definice}
A $1$-form $\alpha$ on $\Sigma$ with values in the set of global smooth sections $\Gamma_{G}(\Sigma,X^{\ast}(T^{\ast}M))$ is a smooth assignment $p \in \Sigma \mapsto \alpha(p)$, where $\alpha(p): T_{p}(\Sigma) \rightarrow \pi'^{-1}(p) \equiv T^{\ast}_{X(p)}(M) $ is a linear map. $\pi'$ denotes the projection of the pullback bundle $X^{\ast}(T^{\ast}M)$.

We define the action of $\alpha$ on a smooth vector field $V \in \mathfrak{X}(\Sigma)$ as
\begin{equation}
\< \alpha, V \>(p) := \< \alpha(p), V_{p} \> \in T_{X(p)}^{\ast}(M),
\end{equation}
for every $p \in \Sigma$, where $V_{p}$ denotes the value of $V$ at $p$. $\< \alpha, V \>$ can be thus interpreted as the global section of the pullback bundle $X^{\ast}(T^{\ast}M)$.

The requirement of the smooth assignment can be then more precisely stated as the smoothness of the section $\< \alpha, V \>$.

Therefore $\alpha$ defines a $C^{\infty}(\Sigma)$-linear map from $\mathfrak{X}(\Sigma)$ to $\Gamma_{G}(\Sigma,X^{\ast}(T^{\ast}M))$.

In the same way we can define a $k$-form on $\Sigma$ with values in $\Gamma_{G}(\Sigma,X^{\ast}(T^{\ast}M))$ and a $k$-form on $\Sigma$ with values in $\Gamma_{G}(\Sigma,X^{\ast}(TM))$, where $X^{\ast}(TM)$ is the pullback bundle of $TM$ by $X$.
\end{definice}

\begin{rem}
Previous definition can be equivalently stated as follows:
$1$-form $\alpha$ on $\Sigma$ with values in $\Gamma_{G}(\Sigma,X^{\ast}(T^{\ast}M))$ is a $C^{\infty}(\Sigma)$-linear map from $\mathfrak{X}(\Sigma)$  to $\Gamma_{G}(\Sigma,X^{\ast}(T^{\ast}M))$.

The reason and proof is similar to the same statement for ordinary $1$-forms $\Omega^{1}(\Sigma)$.
\end{rem}

\begin{definice}
The space of $k$-forms on $\Sigma$ ($k \in \{0,1,2\}$) with values in $\Gamma_{G}(\Sigma,X^{\ast}(T^{\ast}M))$ is denoted as $\Omega^{k}(\Sigma,X^{\ast}(T^{\ast}M))$. In the same way, $\Omega^{k}(\Sigma,X^{\ast}(TM))$ denotes the space of $k$-forms on $\Sigma$ with values in $\Gamma_{G}(\Sigma,X^{\ast}(TM))$.
\end{definice}

\begin{lemma}
Let $(X,A)$ be a vector bundle map $(X,A): T\Sigma \rightarrow T^{\ast}M$.

The map $A$ of the total spaces can be considered as the element of $\Omega^{1}(\Sigma,X^{\ast}(T^{\ast}M))$.
\end{lemma}
\begin{proof}
Let $A$ be the total space map from the vector bundle map $(X,A)$. We can define $\alpha_{A} \in \Omega^{1}(\Sigma,X^{\ast}(T^{\ast}M))$ as
\begin{equation} \<\alpha_{A},V\>(p) \equiv \<\alpha_{A}(p),V_{p}\> := A(V_{p}) \in T_{X(p)}^{\ast}(M) \equiv \pi'^{-1}(p), \end{equation}
where $V \in \mathfrak{X}(\Sigma)$ is a smooth vector field. $\alpha_{A}(p)$ thus indeed maps linearly vectors from $T_{p}(\Sigma)$ to $\pi'^{-1}(p)$.
>From the smoothness of $V$ and of the map $A$ follows, that the resulting section $\<\alpha_{A},V\>$ is smooth. Hence $\alpha_{A} \in \Omega^{1}(\Sigma,X^{\ast}(T^{\ast}M))$.

Conversely, let $\alpha \in \Omega^{1}(\Sigma,X^{\ast}(T^{\ast}M))$. Then for $u \in T\Sigma$ we define the map $A_{\alpha}$ as
\begin{equation} A_{\alpha}(u) := \< \alpha(\pi(u)) , u \>, \end{equation}
where $\pi$ is the projection of the tangent bundle $T\Sigma$. The map $A_{\alpha}$ clearly satisfies the condition
\[ \~\pi \circ A_{\alpha} = X \circ \pi, \]
where $\~\pi$ denotes the projection of the cotangent bundle $T^{\ast}M$ and it is linear in the fibres. The smoothness of $A_{\alpha}$ follows from the smoothness of $\alpha$ and the projection $\pi$. Hence $(X,A_{\alpha})$ is a vector bundle map.
\end{proof}

\begin{rem}
Since now we will not distinguish between $A$ and $\alpha_{A}$, and we will use the notation $A(p) \equiv \alpha_{A}(p)$ for $p \in \Sigma$.
\end{rem}

\begin{rem}
Let $(y^{1},\dots,y^{n})$ be a set of local coordinates on $M$. $A$ can be locally expanded as
\begin{equation}  \label{psm_expansion} A(p) = A_{i}(p) \left. dy^{i} \right|_{\sss{X(p)}}, \end{equation}
for $p \in \Sigma$. $A_{i} \in \Omega^{1}(\Sigma)$ are uniquely determined $1$-forms on $\Sigma$, called the {\bfseries{component $1$-forms}} of $A$.
We will use the notation
\begin{equation} A = A_{i} dy^{i}. \end{equation}
\end{rem}
\begin{rem}
In fact, one has to be very careful, because expansion (\ref{psm_expansion}) has the good sense only for such $p \in \Sigma$, where $X(p)$ stays in the area of $M$, where coordinates $(y^{1},\dots,y^{n})$ are defined. Hence $A_{i}$ are, strictly speaking, not uniquely determined in the "evil points" of $\Sigma$.

However, taking Poisson-Lie groups as target manifolds, we have global frame fields on $M$ - left-(right-)invariant vector fields. This solves our problem in such cases.

Alternatively, we can impose on the map $X$ to not "come out" of the chosen coordinate patch.
\end{rem}
Let us examine for a moment the map $X_{\ast}$ tangent to the map $X: \Sigma \rightarrow M$. Constructed at given point $p \in \Sigma$, it maps linearly
$T_{p}(\Sigma) \rightarrow T_{X(p)}(M)$. Therefore we can define a $1$-form $dX$ on $\Sigma$ with values in $\Gamma_{G} \big( \Sigma,X^{\ast}(TM)\big)$, as:
\[ \<dX,Y\>(p) := X_{\ast}(Y_{p}) \in T_{X(p)}(M), \]
for $Y \in \mathfrak{X}(\Sigma)$. If we write it in the local coordinates $(y^{1},\dots,y^{n})$ on $M$ and local coordinates $(\sigma^{1},\sigma^{2})$ on $\Sigma$, we get
\[ \<dX,Y\>(p) := Y^{\mu}(p) X_{\ast}( \left. \frac{\partial}{\partial \sigma^{\mu}} \right|_{p} ) = Y^{\mu}(p) \left. \frac{\partial X^{i}}{\partial \sigma^{\mu}} \right|_{p} \ddy{i}{X(p)}.\]
Hence
\[ dX(p) = dX^{i}(p) \ddy{i}{X(p)}, \]
where $dX^{i} := X^{\ast}(dy^{i})$. Let us remark that $dX$ can be viewed as a total space map of the bundle map $(X,dX): T\Sigma \rightarrow TM$.

Canonical pairing on $M$ allows us to define the induced pairing of $k$-forms on $\Sigma$ with values in the global sections of pullback bundles $X^{\ast}(TM)$ and $X^{\ast}(T^{\ast}M)$ respectively.

\begin{definice}
Let $A \in \Omega^{k}(\Sigma,X^{\ast}(T^{\ast}M))$, $B \in \Omega^{l}(\Sigma,X^{\ast}(TM))$. $A = A_{i} dy^{i}$, $B = B^{j} \frac{\partial}{\partial y^{j}}$.
Then we define a pairing of $A$ with $B$ as
\begin{equation} \label{psm_pairingdef} \<A,B\>(p) := A_{i}(p) \^ B^{j}(p) \ \< \left. dy^{i} \right|_{\sss{X(p)}}, \left. \frac{\partial}{\partial y^{j}} \right|_{\sss{X(p)}} \hspace{-0.4cm} \> = A_{i}(p) \^ B^{i}(p), \end{equation}
for all $p \in \Sigma$. Hence $\<A,B\> \in \Omega^{k+l}(\Sigma)$.
This definition does not depend on the particular choice of coordinates in $M$.
\end{definice}

Let $p \in \Sigma$ and $\P$ be a Poisson bivector on $M$, Let $A \in
\Omega^{1}(\Sigma,X^{\ast}(T^{\ast}M))$. We define
\[ \io_{A}(\P)(X(p)) := A_{i}(p) \P(X(p))(dy^{i}|_{\sss{X(p)}},\cdot) = A_{i}(p) \P^{ij}(X(p)) \ddy{j}{X(p)}. \]
This is an expansion of some $V \in \Omega^{1}(\Sigma,X^{\ast}(TM))$. We set $\io_{A}(\P)(X) := V$.

\begin{definice}
Let $\Sigma$ be a $2$-dimensional orientable differentiable manifold, let $(M,\P)$ be an $n$-dimensional Poisson manifold. A {\bfseries{Poisson sigma model}} is a field model defined by the action integral
\begin{equation} \label{psm_action} S[X,A] := \int_{\Sigma} \<A,dX\> - \frac{1}{2} \<A,\io_{A}(\P)(X)\>. \end{equation}
Suppose we have the local coordinates $(y^1,\dots,y^n)$ in some neighbourhood $U \subset M$, $\P^{ij} \equiv \P(dy^{i},dy^{j})$. If $X(\Sigma) \subset U$, we can write the action as
\begin{equation} \label{psm_actioncoord} S[X,A] := \int_{\Sigma} A_{i} \^ dX^{i} + \frac{1}{2} \P^{jk}(X) A_{j} \^ A_{k}, \end{equation}
where $A = A_{i} dy^{i}$ and $dX^{i} = X^{\ast}(dy^{i})$.
\end{definice}

\begin{rem}
Poisson sigma models are in many articles (\cite{rakusaci2}, \cite{rakusaci1}) defined by the action (\ref{psm_actioncoord}), with the map $X$ not restricted to $U$. Strictly speaking, this is not correct, because there exist points of $\Sigma$, where the integrand has no sense at all. However, we can always locally write the Lagrangian:
\begin{equation} L[X,A](p) = A_{i}(p) \^ dX^{i}(p) + \frac{1}{2} \P^{jk}(X(p)) A_{j}(p) \^ A_{k}(p),\end{equation}
for such $p \in \Sigma$, where $X(p) \in U$.
\end{rem}

\subsection{Variational principle, equations of motion}
\label{sec_psmvar} In this section we will in detail derive the
equations of motion of the Poisson sigma model, using the action
(\ref{psm_action}) and a variational principle.

It is quite simple to get the equations of motion from (\ref{psm_actioncoord}), putting $\~X^{i} = X^{i} + \epsilon Y^{i}$, $\~A_{i} = A_{i} + \~\epsilon B_{i}$ and using the ordinary per partes trick in the calculation of $S[\~X,\~A] - S[X,A]$. However, this approach is heavily coordinate-dependent, especially in the case, where the form of action \ref{psm_actioncoord} has no real sense. This led us to the following idea.

We will parametrize each variation of the fields $(X,A)$ by infinitesimal constants $\epsilon, \~\epsilon \in \R$, $|\epsilon|,|\~\epsilon| \ll 1$, $1$-form $B \in \Omega^{1}(\Sigma,X^{\ast}(T^{\ast}M))$ and by a smooth vector field $Y \in \mathfrak{X}(M)$ on $M$, such that its local flow $\phi_{\epsilon}^{Y}(X(p))$ is defined for all $p \in \Sigma$. If $\partial \Sigma \neq \emptyset$, we have to impose $(\forall q \in \partial \Sigma) \ (Y(X(q)) = 0)$.

\begin{figure}[htbp]
\begin{center}
\includegraphics[width=90mm]{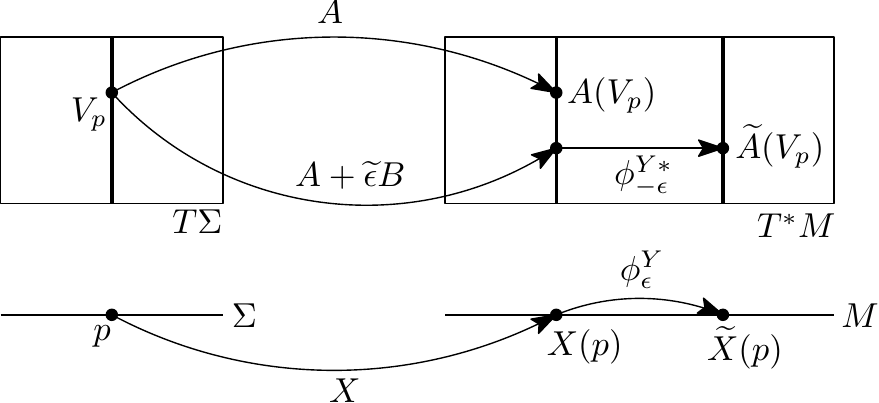}
\caption{Variation of fields $(X,A)$}
\label{fig_variace}
\end{center}
\end{figure}

Shifted fields $\~X$ and $\~A$ are then set as
\begin{equation}
\~X(p) := \phi_{\epsilon}^{Y}(X(p)),
\end{equation}
\begin{equation}
\~A(V_{p}) := \phi_{-\epsilon}^{Y \ast}\big((A + \~\epsilon B)(V_p)\big),
\end{equation}
for all $p \in \Sigma$ and $V_{p} \in T_{p}(\Sigma)$.  For illustration, see the figure \ref{fig_variace}.
It is obvious that $(\~X,\~A)$ is again a vector bundle map $T\Sigma \rightarrow T^{\ast}M$.

Let $(y^{1},\dots,y^{n})$ be a set of local coordinates on $U \subset M$. Vector field $Y$ can be for $x \in U$ expanded as
\begin{equation} Y(x) = Y^{i}(x) \ddy{i}{x}. \end{equation}

$\~A$ can be locally expanded as
\begin{equation} \label{psm_variace_locala} \~A(p) = (A_{i} + \~\epsilon B_{i})(p) \ \phi_{-\epsilon}^{Y \ast} ( \left. dy^{i} \right|_{\sss{X(p)}} ).\end{equation}

For $d\~X$, $p \in \Sigma$ and $V_{p} \in T_{p}(\Sigma)$ we have
\[ \<d\~X(p),V_{p}\> \equiv \~X_{\ast}(V_{p}) = \phi_{\epsilon \ast}^{Y} (X_{\ast}(V_{p})),\]
and thus locally
\begin{equation} \label{psm_variace_localdx}
d\~X(p) = dX^{i}(p) \ \phi_{\epsilon \ast}^{Y}(\ddy{i}{X(p)}).
\end{equation}

We are now ready to proceed with the computation of the first order (in $\epsilon$ and $\~\epsilon$) term in the difference of the Lagrangians $L[\~X,\~A] - L[X,A]$, where
\begin{equation} L[X,A] := \<A,dX\> - \frac{1}{2} \<A,\io_{A}(\P)(X)\>. \end{equation}

>From the definition of pairing (\ref{psm_pairingdef}) and the local
expansions (\ref{psm_variace_locala}) and
(\ref{psm_variace_localdx}) of $\~A$ and $d\~X$ respectively it is
clear, that
\begin{equation} \label{psm_variace_1}
\< \~A, d\~X \> = \<A, dX\> + \~\epsilon \<B, dX\>.
\end{equation}

To deal with the second term of $L[\~X,\~A]$, we should remind that
$\phi_{\epsilon}^{Y}$ is a diffeomorphism of $M$. Therefore for $x
\in U$ we can use $\phi_{\epsilon \ast}^{Y}( \ddy{i}{x} )$ as the
basis of the tangent space $T_{\phi_{\epsilon}^{Y}(x)}(M)$.

Poisson bivector $\P$ can be thus expanded at $\~X(p) \equiv
\phi_{\epsilon}^{Y}(X(p))$ expanded as
\begin{equation} \P(\~X(p)) = \P_{\ast}^{ij}(\~X(p)) \ \phi_{\epsilon \ast}^{Y}(\ddy{i}{X(p)}) \otimes \phi_{\epsilon \ast}^{Y}(\ddy{j}{X(p)}), \end{equation}
where
\[
\P_{\ast}^{ij}(\~X(p) = \P(\~X(p))\big( \phi_{-\epsilon}^{Y \ast}(\left. dy^{i} \right|_{\sss{X(p)}} ), \phi_{-\epsilon}^{Y \ast}(\left. dy^{j} \right|_{\sss{X(p)}}) \big)
\equiv \]
\begin{equation} \label{psm_pstar} \equiv \phi_{-\epsilon \ast}^{Y} \big( \P(\~X(p)) \big) \Big( \left. dy^{i} \right|_{\sss{X(p)}}, \left. dy^{j} \right|_{\sss{X(p)}}\Big).  \end{equation}

In the following we will omit the explicit writing of $(p)$ in every term, but we still mean everything written at the particular point $p \in \Sigma$.
For infinitesimal $\epsilon$ we can rewrite $\P_{\ast}^{ij}(\~X)$ as
\[ \P_{\ast}(\~X)^{ij} = \phi_{-\epsilon \ast}^{Y} \big( \P(\~X) \big) \Big( \left. dy^{i} \right|_{\sss{X}}, \left. dy^{j} \right|_{\sss{X}}\Big) = \]
\[ = \big( \P(X) + \epsilon [\mathcal{L}_{Y}(\P)](X) \big) \Big( \left. dy^{i} \right|_{\sss{X}}, \left. dy^{j} \right|_{\sss{X}}\Big) = \]
\[ = \P^{ij}(X) + \epsilon [\mathcal{L}_{Y}(\P)]^{ij}(X). \]
Then we can write
\[ \io_{\~A}(\P)(\~X) = \P(\~X)(\~A,\cdot) = (A_{i} + \~\epsilon B_{i}) \ \P(\~X)\big( \phi_{-\epsilon}^{Y \ast}( \left. dy^{i} \right|_{\sss{X}} ) , \cdot \big) = \]
\[ = (A_{i} + \~\epsilon B_{i}) \P_{\ast}^{ij}(\~X) \ \phi_{\epsilon \ast}^{Y}(\ddy{j}{X}). \]
Hence
\[ -\frac{1}{2}\<\~A,\io_{\~A}(\P)(\~X)\> = -\frac{1}{2} (A_{k} + \~\epsilon B_{k}) \^ (A_{i} + \~\epsilon B_{i}) \P_{\ast}^{ik}(\~X) = \]
\[ = -\frac{1}{2} A_{k} \^ A_{i} \P^{ik}(X) - \~\epsilon B_{k} \^ A_{i} \P^{ik} - \epsilon \frac{1}{2} A_{k} \^ A_{i} [\mathcal{L}_{Y}(\P)]^{ik}, \]
where we omitted the second order terms in the infinitesimal parameters $\epsilon$ and $\~\epsilon$.
Therefore finally
\begin{equation} \label{psm_variace_2}
-\frac{1}{2}\<\~A,\io_{\~A}(\P)(\~X)\> = -\frac{1}{2}\<A,\io_{A}(\P)(X)\> - \~\epsilon \<B,\io_{A}(\P)(X)\>  -\epsilon \frac{1}{2} \<A,\io_{A}[\mathcal{L}_{Y}(\P)](X)\>.\end{equation}
Putting together (\ref{psm_variace_1}) and (\ref{psm_variace_2}) we get
\begin{equation} \label{psm_variace_rozdil}
L[\~X,\~A] - L[X,A] = \~\epsilon \<B,dX - \io_{A}(\P)(X)\> - \epsilon \frac{1}{2} \<A,\io_{A}[\mathcal{L}_{Y}(\P)](X)\>
\end{equation}

We will state and prove a following lemma to proceed:
\begin{lemma} Let
$\~Y \in \Omega^{0}(\Sigma,X^{\ast}(TM))$ is defined for $p \in
\Sigma$ as
\begin{equation} \label{psm_variace_ytilda}
\~Y(p) := Y(X(p)).
\end{equation}
Its local expansion is $\~Y(p) = Y^{i}(X(p)) \ddy{i}{X(p)} \equiv
\~Y^{i}(p) \ddy{i}{X(p)}$, where we denote $\~Y^{i}(p) :=
Y^{i}(X(p))$. It satisfies (in the first order in $\epsilon$)
\begin{equation} \label{psm_variace_lemma} d\~Y^{i} = {Y^{i}}_{,m}(X) dX^{m}. \end{equation}
\end{lemma}
\begin{proof}
Expansion of $\~X = \phi_{\epsilon}^{Y}(X)$ in local coordinates on $M$ reads
\[ \~X^{i} = X^{i} + \epsilon Y^{i}(X) = X^{i} + \epsilon \~Y^{i}. \]
Hence
\begin{equation} \label{psm_variace_lemmaeq} d\~X^{i} = dX^{i} + \epsilon d\~Y^{i}. \end{equation}
On the other side, as we know, $d\~X^{i}$ constitute the component $1$-forms of $d\~X \in \Omega^{1}(\Sigma,\~X^{\ast}(TM))$. We can proceed from (\ref{psm_variace_localdx}):
\[ d\~X = dX^{i} \ \phi_{\epsilon \ast}^{Y}(\ddy{i}{X}) = dX^{i} \Big( \ddy{i}{\~X} - \epsilon [\mathcal{L}_{Y}( \frac{\partial}{\partial y^{i}} )](\~X) \Big) =\]
\[ = dX^{i} \ddy{i}{\~X} + \epsilon dX^{i} \ {Y^{m}}_{,i}(\~X) \ddy{m}{\~X} = dX^{i} \ddy{i}{\~X} + \epsilon dX^{i} \ {Y^{m}}_{,i}(X) \ddy{m}{\~X} =\]
\[ = (dX^{i} + \epsilon {Y^{i}}_{,m}(X) dX^{m}) \ddy{i}{\~X}. \]
Comparison with (\ref{psm_variace_lemmaeq}) gives us (\ref{psm_variace_lemma}).
\end{proof}

We can now step to the derivation of the equations of motion. We impose the condition of extremality for $(X,A)$, that is
\begin{equation} S[\~X,\~A] - S[X,A] = 0. \end{equation}
>From (\ref{psm_variace_rozdil}) this is equivalent to
\[ \int_{\Sigma} \~\epsilon \<B,dX - \io_{A}(\P)(X)\> - \epsilon \frac{1}{2} \<A,\io_{A}[\mathcal{L}_{Y}(\P)](X)\> = 0. \]
By putting $\epsilon = 0$ (variation of the $1$-form $A$ only), we get the first equation of the motion in the form
\begin{equation} \label{psm_eqm1a} dX = \io_{A}(\P)(X), \end{equation}
or in the local coordinates as
\begin{equation} \label{psm_eqm1b} dX^{i} + \P^{ij}(X)A_{j} = 0. \end{equation}

For the analysis of the second term we can (and have to) use the
equation (\ref{psm_eqm1b}):
\[ -\frac{1}{2} \<A, \io_{A}(\mathcal{L}_{Y}(P))(X)\> = \frac{1}{2} A_{i} \^ A_{k} \ [\mathcal{L}_{Y}(\P)]^{ik}(X) = \]
\[ = \frac{1}{2} A_{i} \^ A_{k} \Big( Y(\P^{ik})(X) - {Y^{i}}_{,m}(X) \P^{mk}(X) - {Y^{k}}_{,m}(X) \P^{im}(X) \Big) = \]
\[ = \frac{1}{2} A_{i} \^ A_{k} Y(\P^{ik})(X) + \P^{mi}(X) A_{i} \^ A_{k} {Y^{k}}_{,m}(X) \stackrel{(\ref{psm_eqm1b})}{=} \]
\[ \stackrel{(\ref{psm_eqm1b})}{=} \frac{1}{2} A_{i} \^ A_{k} \~Y^{m} {\P^{ik}}_{,m}(X) - dX^{m} {Y^{k}}_{,m}(X) \^ A_{k} \stackrel{(\ref{psm_variace_lemma})}{=} \]
\[ \stackrel{(\ref{psm_variace_lemma})}{=} \frac{1}{2} A_{i} \^ A_{k} \~Y^{m} {\P^{ik}}_{,m}(X) - d\~Y^{k} \^ A_{k} = \]
\[ = \~Y^{m} \Big( dA_{m} + \frac{1}{2} {\P^{ik}}_{,m}(X) A_{i} \^ A_{k} \Big) - d (\~Y^{k} A_{k}). \]
The boundary term is clearly coordinate invariant and it vanishes under integration. The first term is not coordinate dependent for $(X,A)$ satisfying (\ref{psm_eqm1b}), which is enough for the derivation of the extremal equation. Both terms can be thus with no problems integrated. Hence we get the second equation (we can choose the vector field $Y$ (almost) arbitrarily):
\begin{equation} dA_{k} + \frac{1}{2} {\P^{ij}}_{,k}(X) A_{i} \^ A_{j} = 0.\end{equation}

This equation cannot be written globally in general, but it transforms itself well for $(X,A)$ solving the first equation. We will again find various possibilities for linear Poisson sigma models or Poisson-Lie sigma models, where we can use more global structures. We can sum up the preceding text in the following proposition \cite{rakusaci1}:

\begin{tvrz}
The extremal fields $(X,A)$ of the Poisson sigma model given by the action (\ref{psm_action}) have to satisfy the equations written locally as
\begin{equation} \label{psm_eqm1} dX^{i} + \P^{ij}(X) A_{j} = 0, \end{equation}
\begin{equation} \label{psm_eqm2} dA_{k} + \frac{1}{2} {\P^{ij}}_{,k}(X) A_{i} \^ A_{j} = 0. \end{equation}
\end{tvrz}

\begin{rem}
Using a Poisson bivector $\P$, one can define a vector bundle map $(\#,Id_{M})$, where $\#: T^{\ast}M \rightarrow TM$ is given as
\begin{equation} \#(\alpha) := \P(m)(\alpha,\cdot), \end{equation}
for $\alpha \in T_{m}^{\ast}(M)$, $m \in M$. $\#$ is usually called a Poisson anchor.

Equation (\ref{psm_eqm1}) can be then equivalently stated as
\begin{equation} \label{psm_eqm1bundle} \# \circ A = dX, \end{equation}
where $dX: T\Sigma \rightarrow TM$ is a vector bundle map over $X$, as is $\# \circ A$.
\end{rem}

\subsection{Linear Poisson sigma model} \label{sec_psmlinear}
In this section we will introduce a notation and the form of
equations (\ref{psm_eqm1}) and (\ref{psm_eqm2}) for a linear Poisson
sigma model.

Let $\g$ be a finite-dimensional real Lie algebra with a Lie bracket
$[\cdot,\cdot]$. We can consider its dual space $\g^{\ast}$ as a
differentiable manifold.

We can use an arbitrary (but fixed) basis $(T_i)_{i=1}^n$ of $\g$ as
the (global) coordinates on $\g^{\ast}$. The tangent space
$T_{\xi}(\g^{\ast})$ at any point $\xi \in \g^{\ast}$ can be
identified with $\g^{\ast}$ itself, there exists the linear
isomorphism $\mathbf{A}_{\xi}: \g^{\ast} \rightarrow
T_{\xi}(\g^{\ast})$, such that $\mathbf{A}_{\xi}(T^{i}) =
\ddTdole{i}{\xi}$, where $(T^{i})$ is a basis of $\g^{\ast}$ dual to
$(T_i)$.

We get the same map $\mathbf{A}_{\xi}$ for every basis of $\g$, that
is if $(Y_i)_{i=1}^n$ is also a basis of $\g$, then
$\mathbf{A}_{\xi}(Y^{i}) = \left. \frac{\partial}{\partial{Y_i}}
\right|_{\xi}$. We can then define a Poisson bracket of $f,g \in
C^{\infty}(\g^{\ast})$ as

\begin{equation} \{f,g\}(\xi) := \< \xi , [\mathbf{A}_{\xi}^{\ast}((df)_{\xi}),\mathbf{A}_{\xi}^{\ast}((dg)_{\xi})] \> ,\end{equation}
for all $\xi \in \g^{*}$. $A_{\xi}^{\ast}: T_{\xi}^{\ast}(\g^{\ast})
\rightarrow \g$ is the linear map dual to $A_{\xi}$.

One should easily check, that for arbitrary smooth functions $f,g$
on $\g^{*}$ and $\xi \in \g^{*}$
\begin{equation} \label{pb_example1}
\{f,g\}(\xi) = \P_{ij}(\xi) \left. \frac{\partial f}{\partial
T_i}\right|_{\xi} \left. \frac{\partial g}{\partial
T_j}\right|_{\xi},
\end{equation}
where
\begin{equation}
\P_{ij}(\xi) = {c_{ij}}^{k} \< \xi, T_k \> \equiv {c_{ij}}^{k} \xi_k
\end{equation}
and $c_{ij}^{k}$ are the structure coefficients of $\g$ with respect
to $(T_i)_{i=1}^n$, i.e.
\begin{equation}
[T_i,T_j] = {c_{ij}}^{k} T_k.
\end{equation}
A skew-symmetry and a derivation property of the Poisson bracket can
be seen from the definition. Coordinate expression of the Jacobi
identities reads
\begin{equation}
({c_{jk}}^{r} {c_{ri}}^{s} + {c_{ki}}^{r} {c_{rj}}^{s} +
{c_{ij}}^{r} {c_{rk}}^{s})\< \xi , T_s \> = 0,
\end{equation}
which is true because of the Jacobi identities for the Lie bracket
on $\g$. The dual space $\g^{*}$ with the Poisson bracket
(\ref{pb_example1}) is then a Poisson manifold. For obvious reasons
it is called {\bfseries{linear Poisson structure}} (components of
$\P$ are linear functions on $\g^{\ast}$).

\begin{definice}A coadjoint representation $ad^{\ast}$ of $\g$ on $\g^{\ast}$ is defined by
\begin{equation} \<ad_{Y}^{\ast}(\xi),Z\> := -\<\xi,ad_{Y}(Z)\>,\end{equation}
for all $Y,Z \in \g$ and $\xi \in \g^{\ast}$, where $\<\cdot,\cdot\>$ denotes the canonical pairing on $\g$.
\end{definice}

Let us remind a useful property of the Killing form $K$ of
semisimple Lie algebras.

\begin{lemma}
Let $\g$ be a finite-dimensional semisimple Lie algebra. Then the inverse of the Killing form $K^{-1}$ is $ad^{\ast}$-invariant, that is
\begin{equation} \label{lem_killing} K^{-1}(ad_{X}^{\ast}(\xi),\eta) + K^{-1}(\xi,ad_{X}^{\ast}(\eta)) = 0,\end{equation}
$\forall \xi,\eta \in \g^{\ast}$ and $\forall X \in \g$.
\end{lemma}

Let $(T_1,\dots,T_n)$ be an arbitrary chosen basis of $\g$.
We will use it as global coordinates on $\g^{\ast}$. Let $A = A^{i} dT_{i}$, $A \in \Omega^{1}(\Sigma,X^{\ast}(T^{\ast}\g^{\ast}))$. We will see that equations of motion can be in this case written in a simple, coordinate-free way.

We define a $1$-form $\~A$ on $\Sigma$ with values in Lie algebra $\g$:
\begin{equation} \~A := A^{i} T_{i}. \end{equation}
It is clear that $\~A$ does not depend on the choice of the basis $(T_{i})_{i=1}^{n}$.

We can consider $X$ as the $0$-form on $\Sigma$ with values in $\g^{\ast}$, that is
\begin{equation} \label{psm_linXdef} X(p) = T_{i}(X(p)) T^{i} \equiv X_{i}(p) T^{i}, \end{equation}
for all $p \in \Sigma$.

To avoid confusion: $dX$ in the following proposition stays for $1$-form with values in $\g^{\ast}$, obtained as the exterior
derivative of $0$-form $X$ defined by (\ref{psm_linXdef}).

This point of view allows us to state the following proposition:

\begin{tvrz}
Suppose we have a linear Poisson sigma model and notation described above. The equations of motion (\ref{psm_eqm1}) and (\ref{psm_eqm2}) can be then written in the form
\begin{equation} \label{psm_eqm1lin} dX + ad_{\~A}^{\ast}(X) = 0, \end{equation}
\begin{equation} \label{psm_eqm2lin} d\~A + \frac{1}{2} [\~A \^ \~A]_{\g} = 0. \end{equation}
\end{tvrz}

Moreover, if $\g$ is a semisimple Lie algebra, we can rewrite (\ref{psm_eqm1lin}) as
\begin{equation} \label{psm_eqm1linss} d\~X + [\~A,\~X]_{\g} = 0, \end{equation}
where $\~X := K^{-1}(X,\cdot)$ is a $0$-form on $\Sigma$ with values in $\g$ and $K^{-1}$ is the inverse of the Killing form.

\begin{proof}
For the linear Poisson structure, we have $\P_{ij}(\xi) = {c_{ij}}^{k} \xi_{k}$ and $\left. \frac{\partial{\P_{ij}}}{\partial T_k}\right|_{\xi} = {c_{ij}}^{k}$, where $[T_i,T_j] = {c_{ij}}^{k} T_k$. Therefore the equations (\ref{psm_eqm1}) take the form
\begin{equation} dX_{i} + X_{k} {c_{ij}}^{k} A^{j} = 0. \end{equation}
Now we will compute $\<T_i,dX + ad_{\~A}^{\ast}(X)\>$, where $\<\cdot,\cdot\>$ is the canonical pairing on $\g$:
\[ 0 = \<T_i,dX + ad_{\~A}^{\ast}(X)\> = dX_i + \<T_i,ad_{\~A}^{\ast}(X)\> = dX_i - \<ad_{\~A}(T_i),X\> = \]
\[ = dX_i - A^{j} X_k \<ad_{T_j}(T_i),T^k\> = dX_i + A^{j} X_k {c_{ij}}^{k}. \]
Thus the equation (\ref{psm_eqm1lin}) is equal to the set of equations (\ref{psm_eqm1}). Let us proceed to the second set of equations. If we again use the properties of $\P$ of the linear Poisson structure, we obtain (\ref{psm_eqm2}) in the form
\begin{equation} dA^i + \frac{1}{2} {c_{jk}}^{i} A^{j} \^ A^{k} = 0. \end{equation}
But ${c_{jk}}^{i} A^{j} \^ A^{k} = \<T^{i}, (A^{j} \^ A^{k}) \ [T_j,T_k]_{\g} \> \equiv \<T^{i}, [\~A \^ \~A]_{\g} \>$.

To obtain the last part of the statement, we take arbitrary $\xi \in \g^{\ast}$ and using the equation (\ref{psm_eqm1lin}) we write
\[0 = K^{-1}(dX + ad_{\~A}^{\ast}(X),\xi) = K^{-1}(dX,\xi) + K^{-1}(ad_{\~A}^{\ast}(X),\xi). \]
For the first term we have straight from the definition $K^{-1}(dX,\xi) \equiv \<\xi,d\~X\>$.
We apply lemma \ref{lem_killing} to rewrite the second term:
\[ K^{-1}(ad_{\~A}^{\ast}(X),\xi) = -K^{-1}(X,ad_{\~A}^{\ast}(\xi)) \equiv - \<ad_{\~A}^{\ast}(\xi),\~X\> = \<\xi,ad_{\~A}(\~X)\> = \<\xi,[\~A,\~X]_{\g}\>.\]
Hence $\<\xi,d\~X + [\~A,\~X]_{\g}\> = 0$ for all $\xi \in \g^{\ast}$, and we get $d\~X + [\~A,\~X]_{\g} = 0$.
\end{proof}

\begin{rem}
In fact, we can rewrite the action (\ref{psm_action}) using the $1$-form $\~A$ as
\begin{equation} S[X,\~A] = \int_{\Sigma} \<X,(d\~A + \frac{1}{2} [\~A \^\~A])\>_{\g}. \end{equation}

Equation $(\ref{psm_eqm2lin})$ then follows immediately.
\end{rem}
An interesting fact to observe is that equation (\ref{psm_eqm2lin}) does not involve the field $X$ at all.

\section{Poisson-Lie sigma models} \label{sec_psm}
We would be most interested in Poisson sigma models, where target
Poisson manifold is moreover a Poisson-Lie group. We shall call such
models Poisson-Lie sigma models. In next subsections there are given
the basics of the theory of Poisson-Lie groups and notation which we
will extensively use in the following.

\subsection{Poisson-Lie groups, Lie bialgebras}
\begin{definice}
A Lie group $G$, which is also a Poisson manifold, is called a {\bfseries{Poisson-Lie group}}, if the group multiplication map $\mu: G \times G \rightarrow G$ is a Poisson map, considering $G \times G$ endowed with the product Poisson structure.
\end{definice}
This property is equivalent to the multiplicativity of the
corresponding Poisson bivector $\P$, that means that
\begin{equation} \label{multiplicativity} \P(gh) = L_{g \ast}(\P(h)) + R_{h \ast}(\P(g)), \end{equation}
for all $g,h \in G$.

Note that $\P(e) = 0$, that is non-trivial Poisson-Lie group is
never symplectic, not even of constant rank. The simplest example of
Poisson-Lie group is a linear Poisson structure, considered as
Abelian group under addition.

Poisson structure $\P$ on $G$ induces an additional algebraic
structure on its Lie algebra $\g$. It turns out that it is a
$1$-cocycle of Chevalley-Eilenberg Lie algebra cohomology with
respect to the adjoint representation of $\g$ on $\g \otimes \g$,
which induces the second Lie algebra structure on the dual vector
space $\g^{\ast}$.

\begin{definice}
{\bfseries{Lie bialgebra}} $(\g,\delta)$ is a Lie algebra $\g$ equipped with an additional structure, a linear map $\delta : \g \rightarrow \gxg$, such that
\begin{enumerate}
\item $\delta^{\ast} : \g^{\ast} \otimes \g^{\ast} \rightarrow \g^{\ast}$ is a Lie bracket on $\g^{\ast}$,
\item $\delta$ is a $1$-cocycle of $\g$ with values in $\gxg$, i.e. for every $X,Y \in \g$
\begin{equation} \label{bi_def_cocycle}
\Delta(\delta)(X,Y) \equiv ad_X^{(2)}\delta(Y) - ad_Y^{(2)}\delta(X) - \delta([X,Y]) = 0,
\end{equation}
where $ad_{X}^{(2)}(Y \otimes Z) := ad_{X}(Y) \otimes Z + Y \otimes ad_{X}(Z)$, for all $X,Y,Z \in \g$.
\end{enumerate}
The linear map $\delta$ is usually called a {\bfseries{cocommutator}} on $\g$. $\delta^{\ast}$ denotes the transposition with respect to the canonical pairing. Let us denote the Lie bracket on $\g^{\ast}$ as $[\cdot,\cdot]_{\g^{\ast}}$.
\end{definice}

In the following we will extensively use an equivalent description
of Lie bialgebra structure, called Manin triple. The $1$-cocycle
condition is translated into the Jacobi identities for Lie bracket.
It is known that Lie bialgebras are in one to one correspondence
with Manin triples.

\begin{definice}
A {\bfseries{Manin triple}} $(\d,\g,\~\g)$, is a triple of Lie algebras $\d$, $\g$, $\~\g$, such that $\d = \g \oplus \~\g$ as vector spaces, $\g, \~\g$ are Lie subalgebras of $\d$ and as vector subspaces they are isotropic with respect to a non-degenerate, symmetric, ad-invariant bilinear form $\<\cdot,\cdot\>_{\d}$ on $\d$.
\end{definice}

We should now introduce a notation of a convenient basis, which we would use in the following. From the properties of Manin triple it can be easily shown that $\dim(\g) = \dim(\~\g)$. We denote $n = \dim(\g)$. Moreover, for given basis $(T_{i})_{i=1}^{n}$ of $\g$ we can uniquely choose the basis $(\~T^{j})_{j=1}^{n}$ of $\~\g$, such that $(T_{i},\~T^{j})_{i,j=1}^{n,n}$ is a basis of $\d$ and
\begin{equation} \label{canform} \<T_{i},\~T^{j}\>_{\d} = {\delta_{i}}^{j}. \end{equation}
Other combinations vanish due to the isotropy of subalgebras $\g$
and $\~\g$. In the following we will always use a basis of $\d$ of
such form.

If we denote $[T_{i},T_{j}] = {c_{ij}}^{k} T_{k}$ and
$[\~T^{i},\~T^{j}] = {f^{ij}}_{k} \~T^{k}$ the structure constants
of $\g$ and $\~\g$, one could obtain from isotropy and ad-invariance
of $\<\cdot,\cdot\>_{\d}$ that
\begin{equation} \label{psm_mixedrelations} [T_{i},\~T^{j}] = {f^{jk}}_{i} T_{k} - {c_{ik}}^{j} \~T^{k}. \end{equation}
The structure constants ${f^{ij}}_{k}$ of $\~\g$ are same as the structure constants of Lie algebra $\g^{\ast}$ of the corresponding Lie bialgebra, written in the basis $(T^{j})_{j=1}^{n}$ dual to $(T_{i})_{i=1}^{n}$.

The Jacobi identities of $\d$ are equivalent to that of $\g$, $\g^{\ast}$ and the $1$-cocycle condition (\ref{bi_def_cocycle}).

If there exists $r \in \g \otimes \g$, such that Lie biagebra cocommutator $\delta$ can be written as $\delta = ad^{(2)}(r) \equiv \Delta(r)$, the resulting Lie bialgebra is called coboundary and $r$ is called an $r$-matrix. Our approach to the Poisson bivector construction can be used for general Lie bialgebra $(\g,\delta)$ and we thus do not need to discuss the coboundary-ness furthermore.

\begin{definice}
Let $\P$ be a multiplicative Poisson bivector. Its intrinsic derivative $D\P: \g \rightarrow \g \otimes \g$ is defined as
\[ D\P(X) := [\mathcal{L}_{\bar{X}}(\P)](e), \]
for all $X \in \g$, where $\bar{X}$ denotes arbitrary vector field extension of $X$.
\end{definice}

\begin{tvrz}
Let $(G,\P)$ be a Poisson-Lie group. Then the intrinsic derivative $D\P$ of $\P$ defines a Lie bialgebra structure on $\g$. Lie bialgebra $(\g,D\P)$ is called a {\bfseries{tangent Lie bialgebra}} to $(G,\P)$.
\end{tvrz}

For proof of this classical proposition see \cite{luweinstein}. We
would expand the $1$-form $A$ not as $A = A_{\alpha} dy^{\alpha}$
(let us denote the coordinate indices in Greek letters for now), but
rather as $A = A_{k} R_{T^k}$, where $R_{T^{k}}$ are the
right-invariant $1$-forms on $G$ dual to $R_{T_{m}}$ frame fields.
We denote $T^{i}_{X}(p) := X^{\ast}\big(R_{T^i}(X(p))\big) \in
T_{p}^{\ast}(\Sigma)$.

\subsection{Equations of motion}
In the following we rewrite the equations of motion in components with respect to the right-invariant frame fields.
Then we will use the properties of multiplicative bivector fields to calculate the action of right-invariant vector fields
on the Poisson bivector components. Using this we will then derive an intrinsic form of the equations of motion.

\begin{lemma}
The equations of motion rewritten in the components with respect to
the right-invariant frame fields take the form
\begin{equation} \label{psm_eqm1pla} T^{i}_X + \Pi^{ij}(X) A_j = 0, \end{equation}
\begin{equation} \label{psm_eqm2pla} dA_{k} + \frac{1}{2} R_{T_k}(\Pi^{ij})(X) A_{i} \^ A_{j} + {c_{kj}}^{i} A_{i} \^ T_{X}^{j} = 0, \end{equation}\
where $\P = \frac{1}{2} \Pi^{ij} R_{T_{i}} \^ R_{T_{j}}$ and
$[T_{i},T_{j}] = {c_{ij}}^{k} T_{k}$.
\end{lemma}
\begin{proof}
We start from (\ref{psm_eqm1}) and (\ref{psm_eqm2}) written in the
local coordinates $(y^{1},\dots,y^{n})$. We denote the coordinate
indices with Greek letters and right-invariant basis indices with
Latin letters.

\begin{equation} \label{psm_eqm1lem} dX^{\alpha} + \P^{\beta \gamma}(X) A_{\gamma} = 0, \end{equation}
\begin{equation} \label{psm_eqm2lem} dA_{\alpha} + \frac{1}{2} {P^{\beta \gamma}}_{,\alpha}(X) A_{\beta} \^ A_{\gamma} = 0. \end{equation}

We define the $n \times n$ matrix $e$ of functions on $\Sigma$ as
\begin{equation} \label{psm_ematicedef} \ddy{\alpha}{X(p)} = {e{^{k}}}_{\alpha}(p) \ R_{X(p) \ast}(T_k) \end{equation}
and $f(p) := e^{-1}(p)$, for every $p \in \Sigma$.

If $A_{k}$ denotes the component $1$-forms from expansion $A = A_{k}
R_{T^{k}}$, we get
\[ A_{\alpha} = {e^{k}}_{\alpha} A_{k},\]
and for other involved objects
\[ dX^{\alpha} = {f^{\alpha}}_{k} T_{X}^{k}, \ \P^{\beta \gamma}(X) = {f^{\beta}}_{m} {f^{\gamma}}_{n} \Pi^{mn}(X). \]
Equation (\ref{psm_eqm1lem}) can be thus written in the form
\[ T_{X}^{i} + \Pi^{ij}(X) A_{j} = 0. \]
We can now deal with the second equation. First we rewrite the term
$dA_{\alpha}$:
\[ dA_{\alpha} = d\big( {e^{k}}_{\alpha} A_{k} \big) = {de^{k}}_{\alpha} \^ A_{k} + {e^{k}}_{\alpha} dA_{k} = \]
\begin{equation} \label{psm_rinvlem_dpart} = \frac{\partial ( {e^{k}}_{\alpha} )}{ \partial y^{\beta} } dX^{\beta} \^ A_{k} + {e^{k}}_{\alpha} dA_{k}, \end{equation}
where by $\frac{\partial}{ \partial y^{\beta} }$ we always mean
$\left. \frac{\partial}{ \partial y^{\beta} } \right|_{\sss{X}}$.
The second term in (\ref{psm_eqm2lem}) reads
\[ \frac{1}{2} {\P^{\beta \gamma}}_{,\alpha}(X) A_{\beta} \^ A_{\gamma} = \frac{1}{2} \frac{\partial}{\partial y^{\alpha}} \Big( {f^{\beta}}_{m} {f^{\gamma}}_{n} \Pi^{mn}(X) \Big) {e^{a}}_{\beta} {e^{b}}_{\gamma} A_{a} \^ A_{b} = \]
\[ = \frac{1}{2} {e^{k}}_{\alpha} R_{T_k} \Big( {f^{\beta}}_{m} {f^{\gamma}}_{n} \Pi^{mn}(X) \Big) {e^{a}}_{\beta} {e^{b}}_{\gamma} A_{a} \^ A_{b} = \]
\[ = \frac{1}{2} {e^{k}}_{\alpha} {f^{\beta}}_{m} {f^{\gamma}}_{n} R_{T_k}(\Pi^{mn})(X) {e^{a}}_{\beta} {e^{b}}_{\gamma} A_{a} \^ A_{b} + \]
\[ + {e^{k}}_{\alpha} R_{T_k}({f^{\beta}}_{m})(X) {f^{\gamma}}_{n} \Pi^{mn}(X) {e^{a}}_{\beta} {e^{b}}_{\gamma} A_{a} \^ A_{b} =  \]
\[ = \frac{1}{2} {e^{k}}_{\alpha} R_{T_k}(\Pi^{mn})(X) A_m \^ A_n + {e^{k}}_{\alpha} {e^{a}}_{\beta} R_{T_k}({f^{\beta}}_{m})(X) \Pi^{mb} A_{a} \^ A_{b} = \otimes. \]
Using the equation (\ref{psm_eqm1pla}), we can write
\begin{equation} \label{psm_rinvlem_spart} \otimes = \frac{1}{2} {e^{k}}_{\alpha} R_{T_k}(\Pi^{mn})(X) A_m \^ A_n - {e^{k}}_{\alpha} {e^{a}}_{\beta} R_{T_k}({f^{\beta}}_{m})(X) A_{a} \^ T_{X}^{m}. \end{equation}
To continue, we have to use the following trick
\[ R_{T_k}({f^{\beta}}_{m})(X) = R_{T_k}(R_{T_m}(y^{\beta}))(X) = R_{T_m}(R_{T_k}(y^{\beta}))(X) + [R_{T_k},R_{T_m}](y^{\beta})(X) =\]
\[ = R_{T_m}(R_{T_k}(y^{\beta}))(X) - c_{km}^{l} R_{T_l}(y^{\beta})(X) = R_{T_m}({f^{\beta}}_{k})(X) - {c_{km}}^{l} {f^{\beta}}_{l}. \]
Using this, we can rewrite the second term in
(\ref{psm_rinvlem_spart}) as
\[ - {e^{k}}_{\alpha} {e^{a}}_{\beta} R_{T_k}({f^{\beta}}_{m})(X) A_{a} \^ T_{X}^{m} = \]
\[ - {e^{k}}_{\alpha} {e^{a}}_{\beta} R_{T_m}({f^{\beta}}_{k})(X) A_{a} \^ T_{X}^{m} + {e^{k}}_{\alpha} {e^{a}}_{\beta} {c_{km}}^{l} {f^{\beta}}_{l} A_{a} \^ T_{X}^{m} = \]
\[ - {e^{k}}_{\alpha} {e^{a}}_{\beta} R_{T_m}({f^{\beta}}_{k})(X) A_{a} \^ T_{X}^{m} + {e^{k}}_{\alpha} {c_{km}}^{a} A_{a} \^ T_{X}^{m} = \boxtimes. \]
We can write
\[ {e^{a}}_{\beta} R_{T_m}({f^{\beta}}_{k})(X) = - R_{T_m}({e^{a}}_{\beta})(X) {f^{\beta}}_{k}. \]
Hence
\[ \boxtimes = {e^{k}}_{\alpha} R_{T_m}({e^{a}}_{\beta})(X) {f^{\beta}}_{k} A_{a} \^ T_{X}^{m} + {e^{k}}_{\alpha} {c_{km}}^{a} A_{a} \^ T_{X}^{m} =\]
\[ = R_{T_m}({e^{a}}_{\alpha})(X) A_{a} \^ T_{X}^{m} + {e^{k}}_{\alpha} {c_{km}}^{a} A_{a} \^ T_{X}^{m} = \]
\[ = {f^{\beta}}_{m} \frac{\partial({e^{a}}_{\alpha})}{\partial y^{\beta}} A_{a} \^ T_{X}^{m} + {e^{k}}_{\alpha} {c_{km}}^{a} A_{a} \^ T_{X}^{m} = \]
\[ = - \frac{\partial({e^{k}}_{\alpha})}{\partial y^{\beta}} dX^{\beta} \^ A_{k} + {e^{k}}_{\alpha} {c_{km}}^{a} A_{a} \^ T_{X}^{m}.\]
Putting this back into (\ref{psm_rinvlem_spart}), we obtain
\[ \frac{1}{2} {\P^{\beta \gamma}}_{,\alpha}(X) A_{\beta} \^ A_{\gamma} = \]
\[ = {e^{k}}_{\alpha} \Big(\frac{1}{2} R_{T_k}(\Pi^{mn})(X) A_m \^ A_n + {c_{km}}^{a} A_{a} \^ T_{X}^{m} \Big) - \frac{\partial({e^{k}}_{\alpha})}{\partial y^{\beta}} dX^{\beta} \^ A_{k}. \]
Together with (\ref{psm_rinvlem_dpart}), we get
\[ dA_{\alpha} + \frac{1}{2} {\P^{\beta \gamma}}_{,\alpha}(X) A_{\beta} \^ A_{\gamma} = \]
\[ = {e^{k}}_{\alpha} \Big(dA_{k} + \frac{1}{2} R_{T_k}(\Pi^{mn})(X) A_m \^ A_n + {c_{km}}^{a} A_{a} \^ T_{X}^{m} \Big). \]
Therefore from (\ref{psm_eqm2lem}), finally:
\[ dA_{k} + \frac{1}{2} R_{T_k}(\Pi^{mn})(X) A_m \^ A_n + {c_{km}}^{a} A_{a} \^ T_{X}^{m} = 0, \]
which was to be proved.
\end{proof}
Up to now, this rewriting does not seem to be that much useful.
Moreover, instead of an ordinary partial derivative, we have the
action of the right-invariant field in (\ref{psm_eqm2pla}).
Fortunately, as we we will show in the following lemma, this is no
obstacle at all. We shall prove the following lemma for slightly
more general case, using just the properties of multiplicative
tensor fields.

\begin{lemma} \label{lem_reactionaction} Let $\P$ be a multiplicative bivector field (not necessarily Poisson) on Lie group $G$, i.e. it satisfies (\ref{multiplicativity}). Let $(T_{i})_{i=1}^{n}$ be an arbitrary basis of Lie algebra $\g$ corresponding to $G$. Let $\Pi^{ij}$ denote the components of $\P$ with respect to the right-invariant frame, that is $\P = \Pi^{ij} R_{T_{i}} \otimes R_{T_{j}}$.

Denote $[T_{i},T_{j}] = {c_{ij}}^{k} T_k$ and $D\P(T_{k}) =
{f^{ij}}_{k} T_{i} \otimes T_{j}$. Then
\begin{equation} \label{lem_raction}
R_{T_{k}}(\Pi^{ij}) = {c_{kl}}^{i} \Pi^{lj} - {c_{kl}}^{j} \Pi^{li}
+ {f^{ij}}_{k}.
\end{equation}
\end{lemma}
\begin{proof}
First note (see e.g. \cite{luweinstein}) that Lie derivative of
every multiplicative tensor field along any right-invariant vector
field is always a right-invariant tensor field. Then
\[ \mathcal{L}_{R_{T_{k}}}(\P) = \mathcal{L}_{R_{T_{k}}} ( \Pi^{ij} \ R_{T_{i}} \otimes R_{T_{j}}) =  R_{T_{k}}(\Pi^{ij})  \ R_{T_{i}} \otimes R_{T_{j}} + \]
\[ + \Pi^{ij} \ [R_{T_{k}},R_{T_{i}}] \otimes R_{T_{j}} + \Pi^{ij} \ R_{T_{i}} \otimes [R_{T_{k}},R_{T_{j}}] = \]
\[ = (R_{T_{k}}(\Pi^{ij}) - {c_{kl}}^{i} \Pi^{lj} + {c_{kl}}^{j} \Pi^{li} ) \ R_{T_{i}} \otimes R_{T_{j}}. \]
Therefore
\[ D\P(T_{k}) = R_{g^{-1} \ast}([\mathcal{L}_{R_{T_{k}}}(\P)](g)) = (R_{T_{k}}(\Pi^{ij}) - {c_{kl}}^{i} \Pi^{lj} + {c_{kl}}^{j} \Pi^{li} )(g) T_{i} \otimes T_{j}. \]
Since $\<T_{i} \otimes T_{j}, D\P(T_{k})\> = {f^{ij}}_{k}$, we
finally get
\[ R_{T_{k}}(\Pi^{ij})(g) = {c_{kl}}^{i} \Pi^{lj}(g) - {c_{kl}}^{j} \Pi^{li}(g) + {f^{ij}}_{k}. \]
\end{proof}

\begin{rem}
Note that the preceding lemma holds for arbitrary multiplicative
tensor field $\P$, no matter whether $G$ is connected or not. For
$\P$ a Poisson-Lie group bivector, the numbers ${f^{ij}}_{k}$
constitute the structure constants of Lie algebra dual to $\g$, that
is $[T^{i},T^{j}]_{\g^{\ast}} = {f^{ij}}_{k} T^{k}$.
\end{rem}

Let us observe that if $A = A_{k} R_{T^{k}}$, we may consider
$A_{k}$ as the component $1$-forms of $1$-form $\~A$ on $\Sigma$
with values in Lie algebra $\g^{\ast}$. That is
\begin{equation} \~A(p) := A_{k}(p) T^{k}. \end{equation}\
One can easily verify that this definition does not depend on the
choice of the basis $(T_{i})_{i=1}^{n}$ in $\g$.

Right-invariant Maurer-Cartan $1$-form $\Theta_R$ can be expanded as
\begin{equation} \Theta_R(g) = R_{T^{k}}(g) T_{k}, \end{equation}
for all $g \in G$. We define $\Theta_R^X$ as its pullback by $X$ to
$\Sigma$:
\begin{equation} \Theta_R^X := X^{\ast}(\Theta_R) \equiv T_X^{k} T_k. \end{equation}

Poisson bivector $\P$ induces for each $g \in G$ a linear map $\Pi(g): \g^{\star} \rightarrow \g$, defined as
\begin{equation} \label{piinduced} \< \eta, \Pi(g)(\xi) \> := \< \P(g), R_{\xi}(g) \otimes R_{\eta}(g) \>, \end{equation}
for all $\xi, \eta \in \g^{\ast}$, where $R_{\xi}(g) = R_{g^{-1}}^{\ast}(\xi)$.

Using this definitions, we can rewrite the equations of motion in a
very elegant way:

\begin{tvrz} \label{psm_pl_tvrzmain}
Equations of motion of a Poisson-Lie sigma model can be written in
the coordinate-free form
\begin{equation} \label{psm_eqm1pl}
\Theta_R^{X} = \Pi(X)(\~A),
\end{equation}
\begin{equation} \label{psm_eqm2pl}
d\~A + \frac{1}{2} [\~A \^ \~A]_{\g^{\ast}} = 0,
\end{equation}
where $\Pi(X)$ stands for the map $\Pi(g): \g^{\ast} \rightarrow \g$,
defined by (\ref{piinduced}), taking $g = X(p)$.
\end{tvrz}
\begin{proof}
From  (\ref{psm_eqm1pla}) and (\ref{piinduced}) , we have
\[ \<T^{i}, \Theta_R^{X} \> = A_{j} \Pi^{ji}(X) = A_{j} \<T^{i}, \Pi(X)(T^{j}) \> = \<T^{i}, \Pi(X)(\~A) \>. \]
To derive the second equation, we just put (\ref{lem_raction}) into
(\ref{psm_eqm2pla}) and use (\ref{psm_eqm1pla}):
\[ 0 = dA_{k} + \frac{1}{2} R_{T_k}(\Pi^{ij})(X) A_{i} \^ A_{j} + {c_{kj}}^{i} A_{i} \^ T_{X}^{j} = \]
\[ = dA_{k} + \frac{1}{2} \big( {c_{kl}}^{i} \Pi^{lj}(X) - {c_{kl}}^{j} \Pi^{li}(X)\big) A_{i} \^ A_{j} + \]
\[ + \frac{1}{2} {f^{ij}}_{k} A_{i} \^ A_{j} + {c_{kj}}^{i} A_{i} \^ T_X^{j} = \]
\[ = dA_{k} + \frac{1}{2} {f^{ij}}_{k} A_{i} \^ A_{j} + {c_{kl}}^{i} A_{i} \^ \Pi^{lj}(X) A_{j} + {c_{kj}}^{i} A_{i} \^ T_X^{j} \stackrel{(\ref{psm_eqm1pla})}{=} \]
\[ \stackrel{(\ref{psm_eqm1pla})}{=} dA_{k} + \frac{1}{2} {f^{ij}}_{k} A_{i} \^ A_{j} - {c_{kl}}^{i} A_{i} \^ T_{X}^{l} + {c_{kj}}^{i} A_{i} \^ T_X^{j} = \]
\[ = dA_{k} + \frac{1}{2} {f^{ij}}_{k} A_{i} \^ A_{j}. \]
\end{proof}
We have just found a very interesting result. For general
Poisson-Lie sigma model, the second equation of motion
(\ref{psm_eqm2}) takes the form of "zero curvature" equation for
$\g^{\ast}$-valued $1$-form $\~A$. The most important fact is that the
field $X$ is not anyhow present in the equation
$(\ref{psm_eqm2pl})$. It is the equation for $2$-forms on $\Sigma$
only.

This generalizes the result brought in \cite{spanele}, where $\g$ is
supposed to be semisimple with coboundary Lie bialgebra $(\g,\delta
= \Delta(r))$. We can quickly derive the form of equations presented
there, as is done in the following corollary of the proposition
\ref{psm_pl_tvrzmain}:

\begin{definice}
Let $p \in V \otimes V$ be a bilinear form on vector space $V$. We
denote $\ul{p}: V \rightarrow V^{\ast}$ the induced linear map
defined as
\begin{equation} \<w,\ul{P}(v)\> := P(v,w), \end{equation}
for all $v,w \in V$, where $\<\cdot,\cdot\>$ denotes the canonical
pairing on $V$.
\end{definice}

\begin{cor}

Let $G$ be a Poisson-Lie group with semisimple coboundary tangent
Lie bialgebra $(\g,\delta = \Delta(r))$.

We set $R := \ul{a} \circ \ul{K}$, where $K$ is the Killing form on
$\g$ and $a \in \bigwedge^{2} \g$ is the skew-symmetric part of $r$.

Moreover, we define a $1$-form $B$ on $\Sigma$ with values in $\g$
as
\[ B := \ul{K}^{-1}(\~A) \equiv K^{-1}(\~A,\cdot), \]
where $K^{-1}$ is the inverse of the Killing form $K$. Equations
(\ref{psm_eqm1pl}) and (\ref{psm_eqm2pl}) can be then written in the
form
\begin{equation} \label{psm_ssplsm_eqm1} \Theta_{R}^{X} + \big(R - Ad_{X} R Ad_{X}^{-1}\big)(B) = 0, \end{equation}
\begin{equation} \label{psm_ssplsm_eqm2} dB + \frac{1}{2} [B \^ B]_{R} = 0, \end{equation}
where a Lie bracket $[\cdot,\cdot]_{R}$ on $\g$ is defined as
$[X,Y]_{R} := [R(X),R(Y)]_{\g}$, for all $X,Y \in \g$.
\end{cor}
\begin{proof}
If we expand $K$ and $a$ as $K = K_{ij} T^{i} \otimes T^{j}$ and $a
= a^{ij} T_i \otimes T_j$ respectively, we get the matrix of $R$:
\[ {(R_{\X})^{i}}_{j} = K_{jl} a^{li}. \]
For $1$-form $B$ we can write from definition
\[ B \equiv B^{i} T_{i} = K^{ij} A_{j} T_{i}, \]
where $K^{ij}$ are the components of $K^{-1}$ in the basis $T_{i}
\otimes T_{j}$. Poisson bivector $\P$ on $G$ with coboundary tangent
Lie bialgebra can be written using Sklyanin bracket (see e.g.
\cite{luweinstein}) as
\[ \P(g) = L_{g \ast}(a) - R_{g \ast}(a) = R_{g \ast} \big( Ad_{g}(a) - a \big) = \big( Ad_{g}(a)^{ij} - a^{ij} \big) R_{T_i}(g) \otimes R_{T_j}(g). \]
Thus for the components $\Pi^{ij}$ in right-invariant basis
\[ \Pi^{ij}(g) = Ad_{g}(a)^{ij} - a^{ij}. \]
If we denote ${\mathbf{P}^{i}}_{j} := \<T^{i},Ad_{X}(T_j)\>$, we get
\[ \Pi^{ij}(X) = {\mathbf{P}^{i}}_{k} {\mathbf{P}^{j}}_{l} a^{kl} - a^{ij}. \]
Equation (\ref{psm_eqm1pla}) then reads
\[ 0 = T_{X}^{i} + \Pi^{ij}(X)A_{j} = T_{X}^{i} + \big( {\mathbf{P}^{i}}_{k} {\mathbf{P}^{j}}_{l} a^{kl} - a^{ij} \big) A_{j} = \]
\[ = T_{X}^{i} + \big( {\mathbf{P}^{i}}_{k} {\mathbf{P}^{j}}_{l} a^{kl} - a^{ij} \big) K_{jm} B^{m} = T_{X}^{i} + {(R_{\X})^{i}}_{m} B^m + K_{jm} {\mathbf{P}^{j}}_{l} {\mathbf{P}^{i}}_{k} a^{kl} B^{m} = \otimes.\]
Using the Ad-invariance of $K$, we have $K_{jm} {\mathbf{P}^{j}}_{l}
= K_{lj} {(\mathbf{P}^{-1})^{j}}_{m}$. Hence
\[ \otimes = T_{X}^{i} + {(R_{\X})^{i}}_{m} B^m + K_{lj} a^{kl} {(\mathbf{P}^{-1})^{j}}_{m} {\mathbf{P}^{i}}_{k} B^{m} = \]
\[ = T_{X}^{i} + {(R_{\X})^{i}}_{m} B^m - {\mathbf{P}^{i}}_{k} {(R_{\X})^{k}}_{j} {(\mathbf{P}^{-1})^{j}}_{m}  B^{m} = \]
\[ = \<T^{i}, \Theta_{R}^{X} + \big(R - Ad_{X} R Ad_{X}^{-1}\big)(B) \>. \]
This finishes the proof of (\ref{psm_ssplsm_eqm1}). Note that
\begin{equation} \label{bigRisomorphism} [X,Y]_{R} = \ul{K}^{-1} [\ul{K}(X),\ul{K}(Y)]_{\g^{\ast}}, \end{equation}
for all $X,Y \in \g$, where $[\cdot,\cdot]_{\g^{\ast}}$ is a Lie
bracket on $\g^{\ast}$ induced by Lie bialgebra structure. Second
equation (\ref{psm_ssplsm_eqm2}) follows from (\ref{psm_eqm2pl}) and
(\ref{bigRisomorphism}). Indeed, if we apply $\ul{K}^{-1}$ on both
sides of (\ref{psm_eqm2pl}), we get
\[ 0 = \ul{K}^{-1} \big( d\~A + \frac{1}{2} [\~A \^ \~A]_{\~\g} \big) = dB + \frac{1}{2} \ul{K}^{-1} [\~A \^ \~A]_{\~\g} \stackrel{(\ref{bigRisomorphism})}{=} \]
\[ \stackrel{(\ref{bigRisomorphism})}{=} dB + \frac{1}{2} [B \^ B]_{R}. \]
\end{proof}

\subsection{Construction of Poisson-Lie group structure on Drinfel'd double} \label{sec_construction}
In this subsection we give out an explicit method how to construct a
Poisson-Lie group structure $\P$ on Lie group $G$, such that the
given Lie bialgebra $(\g,\delta)$ is tangent to $(G,\P)$.

This particular method can be found (without proofs) in
\cite{normal}, \cite{loops} and in \cite{ybsigma}. We give out the
sketch of the process. All needed properties can be checked directly
(multiplicativity) or using the properties of multiplicative tensor
fields (see \cite{luweinstein}), like vanishing of the
Schouten-Nijenhuis bracket.

\begin{definice}
A {\bfseries{Drinfel'd double}} $D$ is a connected Lie group
corresponding to Lie algebra $\d$ of given Manin triple
$(\d,\g,\~\g)$. Let use denote $G$ and $\~G$ the connected Lie
subgroups corresponding to the subalgebras $\g$ and $\~\g$
respectively.
\end{definice}

\begin{rem}
Usually Drinfel'd double is assumed to be connected and simply
connected. However, only connectedness is required for the
construction of Poisson bivector. For each $\d$ there may be
therefore several possibilities for $D$.
\end{rem}
\begin{rem}
Note that we do not claim that there is a unique Manin triple for
each Drinfel'd double $D$.
\end{rem}

Let $(\d,\g,\~\g)$ be the given Manin triple, corresponding to Lie bialgebra $(\g,\delta)$. Let $D$ be a corresponding Drinfel'd double with Lie subgroups $G$ and $\~G$.
For each $g \in G$ we define a map $\Pi(g): \~\g \rightarrow \g$ as
\begin{equation} \label{pimap} \Pi(g) = P Ad_{g} \~P Ad_{g^{-1}} \~P, \end{equation}
where $P$ and $\~P$ denote the projectors from $\d$ to $\g$ and $\~\g$ respectively and $Ad_{g}$ is the adjoint representation of $D$ on $\d$.

Note that
\begin{equation} \<X,\Pi(g)(Y)\>_{\d} = - \<\Pi(g)(X),Y\>_{\d}, \end{equation}
for all $g \in G$ and $X,Y \in \~\g$.

\begin{definice} \label{pimatrix}
Let us take the basis $(T_{i},\~T^{j})_{i,j=1}^{n}$ of $\d$, satisfying the condition (\ref{canform}).
We denote the matrix of the map $\Pi(g)$ in the basis $(T_{i})_{i=1}^{n}$ of $\g$ and $(\~T^{j})_{j=1}^{n}$ of $\~g$ as $(\Pi(g))_{\X}$. Note that
\begin{equation} \label{pimatrixexpl} (\Pi(g))_{\X}^{ij} \equiv \<T^{i},\Pi(g)(\~T^{i})\> = \<\~T^{i},\Pi(g)(\~T^{j})\>_{\d}. \end{equation}
(Let us emphasize the difference between canonical pairing
$\<\cdot,\cdot\>$ on $\g$ and bilinear form $\<\cdot,\cdot\>_{\d}$
on $\d$.)
\end{definice}

\begin{rem} 
For $g \in G$ we can write the matrix of the adjoint representation
$Ad_{g^{-1}}$ in the basis (\ref{canform}) as
\begin{equation} \label{Admaticeinv} (Ad_{g^{-1}})_{\X} = \left( \begin{array}{cc} a(g)^{T} & b(g)^{T} \\ 0 & d(g)^T \end{array} \right), \end{equation}
where $a(g)$, $b(g)$, $d(g)$ are $G$-dependent $n \times n$
matrices. The matrix of the map $\Pi(g)$ then has the form
\begin{equation} \label{pibyad} (\Pi(g))_{\X} = b(g)a(g)^{-1}.\end{equation}
\end{rem}

\begin{tvrz}
Let $\Pi(g)$ be the map (\ref{pimap}) and $(\Pi(g))_{\X}$ its matrix (\ref{pimatrixexpl}).
We define a bivector field $\P$ on $G$ as
\begin{equation} \label{pdef} \P(g) := -(\Pi(g))_{\X}^{ij} R_{T_i} \otimes R_{T_j}, \end{equation}
for all $g \in G$, where $R_{T_{i}}(g) = R_{g \ast}(T_{i})$ denote the right-invariant
vector fields generated by $T_{i} \in \g$. Then $\P$ is the unique
Poisson-Lie group structure on $G$, such that Lie bialgebra
$(\g,\delta)$ is tangent to $(G,\P)$.

\begin{equation} \label{p} \P = \Pi^{ij} R_{T_i} \otimes R_{T_j} \equiv \frac{1}{2} \Pi^{ij} R_{T_{i}} \^ R_{T_{j}}. \end{equation}
\end{tvrz}

One can repeat this procedure taking $\~G$ instead of $G$ and interchanging the role of $\g$ and $\~\g$ to get the unique Poisson-Lie group structure on $\~G$. Its tangent Lie bialgebra is dual to $(\g,\delta)$, that is corresponding to Manin triple $(\d,\~\g,\g)$.

\subsection{Example}
We will show the most simple non-linear Poisson-Lie sigma model, such that its tangent Lie bialgebra is not coboundary. We will use the method introduced above to construct a Poisson bivector.

Lie algebra $\g$ is set in the basis $(T_1,T_2)$ as
\[ [T_1,T_2] = T_2. \]
We equip $\g$ with a Lie bialgebra cocommutator $\delta$ in the form
\[ \delta(T_2) := \beta (T_1 \otimes T_2 - T_2 \otimes T_1), \]
where $\beta \in \R - \{0\}$ and $\delta(T_1) = 0$. Such Lie bialgebra is not coboundary and corresponding Lie bracket on $\~g$ has the form 
\[ [\~T^{1},\~T^{2}]_{\g^{\ast}} = \beta \~T^{2}. \]
Other commutation relations in $\d$ can be easily computed using (\ref{psm_mixedrelations}).
In a sufficiently small neighbourhood of the unit of the corresponding Lie group $G$, we can use the parametrization
$g = e^{\alpha_1 T_{1}} e^{\alpha_2 T_{2}}$ and define the local coordinates $(y^{1},y^{2})$ as
\[ y^{k}(e^{\alpha T_{1}} e^{\alpha T_{2}}) := \alpha_{k}. \]
The matrices of adjoint representation of Drinfel'd double $D$ can be then calculated using the adjoint representation of Lie algebra $\d$ and the relation 
\begin{equation} (Ad_{e^{\alpha_1 T_1}e^{\alpha_2 T_2}})_{\X} = \prod_{k=1}^{2} e^{\alpha_{k} (ad_{T_{k}})_{\X}}.  \end{equation}
Then by (\ref{pibyad}) and (\ref{pdef}) one gets the matrix $\Pi^{ij}$ of the components of $\P$ with respect to the right-invariant basis as

\[ \Pi^{ij} = \left( \begin{array}{cc} 0 & \beta e^{y^{1}} y^{2} \\ - \beta e^{y^{1}} y^{2} & 0 \end{array} \right). \]

We denote $X^{i}(p) := y^{i}(X(p))$. Then $T^{i}_{X}(p) = {(e)^{i}}_{k}(p) dX^{k}(p)$, where $e$ is the matrix of functions on $\Sigma$ defined by (\ref{psm_ematicedef}). For our example 
\[ e(p) = \left( \begin{array}{cc} 1 & 0 \\ 0 & e^{X^{1}(p)} \end{array} \right). \]
The equations of motion (\ref{psm_eqm1pl}) and (\ref{psm_eqm2pl}) then take the form
\[ dX^{1} = - \beta e^{X^{1}} X^{2} A_{2}, \]
\[ dX^{2} = \beta e^{X^{1}} X^{2} A_{1}, \]
\[ dA_{1} = 0, \]
\[ dA_{2} + \beta A_{1} \^ A_{2} = 0. \]

\section{Conclusion} The proper coordinate independent framework
for formulation of sigma models on Poisson manifolds $(M,\P)$ is a
vector bundle map $(X,A): T\Sigma \rightarrow T^{\ast}M$. The map
$A$ of the total spaces can be also considered as $1$-form on
$\Sigma$ with values in the set of global smooth sections of the
pullback bundle  $X^{\ast}(T^{\ast}M)$.

Variation of the fields $(X,A)$ can be defined as \[ \~X(p) :=
\phi_{\epsilon}^{Y}(X(p)),
\]
\[
\~A(V_{p}) := \phi_{-\epsilon}^{Y \ast}\big((A + \~\epsilon
B)(V_p)\big),
\]
for all $p \in \Sigma$ and $V_{p} \in T_{p}(\Sigma)$ where
$\epsilon,\ \~\epsilon$ are real parameters, $\phi_{\epsilon}^{Y}$
is the local flow of a smooth vector field $Y$ on $M$ and $B$ is a
one form on $\Sigma$ with values in global sections of the pullback
bundle $X^{\ast}(T^{\ast}M)$. Variational principle with the
variations given above then leads to the well known equations of
sigma models on Poisson manifolds \cite{rakusaci1,spanele}.

Poisson sigma models can be constructed on Poisson-Lie groups
irrespectively if the corresponding bialgebra is coboundary or not.
The Poisson bivector is defined by virtue of the adjoint
representation of the group on the bialagebra and right-invariant
vector fields on the group. The equations of motion then can be
written in the coordinate independent form (\ref{psm_eqm1pl}),
(\ref{psm_eqm2pl}).

\section*{Acknowledgement}
This work was supported by the research plan LC527 of the Ministry
of Education of the Czech Republic and by the Grant Agency of the
Czech Technical University in Prague, Grant No. SGS
10/295/OHK4/3T/14.

\end{document}